\documentclass[12pt]{amsproc}
\usepackage{amsfonts,amssymb,amsmath,amsthm}
\usepackage{comment}
\usepackage[english]{babel}
\newtheorem{theorem}{Theorem}[section]
\newtheorem{corollary}[theorem]{Corollary}
\newtheorem{lemma}[theorem]{Lemma}

\theoremstyle{definition}
\newtheorem{definition}[theorem]{Definition}
\theoremstyle{remark}

\numberwithin{equation}{section}
\newcommand{\R}{\mathbb R}
\newcommand{\C}{\mathbb C}

\newcommand{\F}{\mathbb F}

\newcommand{\gna}{\mathrm{gna}}
\newcommand{\ga}{\mathrm{ga}}
\newcommand{\sna}{\mathrm{sna}}
\newcommand{\ona}{\mathrm{ona}}
\newcommand{\una}{\mathrm{una}}
\newcommand{\suna}{\mathrm{suna}}
\newcommand{\tr}{\mathrm{tr}\,}

\newcommand{\act}{\,\triangleright}

\newcounter{zlist}

  \newcounter{blist} 
  \newenvironment{blist}{\begin{list}{(\alph{blist})}{ 
  \usecounter{blist}\leftmargin2.5em\labelwidth2em\labelsep0.5em 
  \topsep0.6ex  
  \parsep0.3ex plus0.2ex minus0.1ex}}{\end{list}} 
  
  \makeatletter
\@namedef{subjclassname@2020}{\textup{2020} Mathematics Subject Classification}
\makeatother

\title{On matrix Lie affgebras} 
\author[T. Brzezi\'nski]{Tomasz Brzezi\'nski}

\address{
Department of Mathematics, Swansea University, 
Fabian Way,
  Swansea SA1 8EN, U.K.\ \newline 
Faculty of Mathematics, University of Bia{\l}ystok, K.\ Cio{\l}kowskiego  1M,
15-245 Bia\-{\l}ys\-tok, Poland}

\email{T.Brzezinski@swansea.ac.uk}
\thanks{The research of Tomasz Brzeziński is partially supported by the National Science Centre, Poland, grant no. 2019/35/B/ST1/01115.}
\author[K.\ Radziszewski]{Krzysztof Radziszewski}

\address{
Faculty of Mathematics, University of Bia{\l}ystok, K.\ Cio{\l}kowskiego  1M,
15-245 Bia\-{\l}ys\-tok, Poland}

\email{k.radziszewski@uwb.edu.pl}


\begin{document}
\begin{abstract}
Lie brackets or Lie affgebra structures on several classes of affine spaces of matrices are studied. These include general normalised affine matrices, special normalised affine matrices, anti-symmetric and anti-hermitian normalised affine matrices and special anti-hermitian normalised affine matrices. It is shown that, when retracted to the underlying vector spaces, they  correspond to classical matrix Lie algebras: general and special linear, anti-symmetric, anti-hermitian and special anti-hermitian Lie algebras respectively.  
 \end{abstract}
\maketitle

\section{Introduction} \noindent
The study of Lie brackets on affine spaces or {\em Lie affgebras} was initiated in \cite{GraGra:Lie} as a part of the geometric programme of developing a frame-independent formulation of classical mechanics \cite{Tul:fra} in which vector bundles are replaced by bundles with affine fibres \cite{GraGra:av1,GraGra:fra,GraGra:av2}. The definition of a Lie bracket in \cite{GraGra:Lie} relies on the linear structure of the vector space underlying an affine space. This reliance has been shown to be not necessary, leading to the intrinsic or vector-space independent formulation of Lie affgebras \cite{BrzPap:Lie}. From this viewpoint a vector space is an artefact of the definition of an affine space \cite{OstSch:bar,BreBrz:hea}, while any Lie bracket on an affine space can be retracted to the (bilinear) Lie bracket on this vector space. In the present paper we  analyse and extend the class of matrix Lie affgebras introduced in \cite{Brz:spe}. In particular we present equivalent formulation of resulting classes, which allows one to identify quickly classical matrix Lie algebras to which matrix Lie affgebras are retracted by fixing an element in accordance with \cite[Theorem~3.15]{BrzPap:Lie}.

The paper is organised as follows. In Section 2 we recall the intrinsic definitions of an affine space and a Lie affgebra. Next we describe the process of retracting of a bi-affine Lie bracket into a bi-linear Lie bracket on an underlying vector space. Section 3 contains main results. First several classes of  matrix Lie affgebras are introduced.  These are then shown to be isomorphic with Lie affgebras which in a natural way relate to classical matrix Lie algebras. As a corollary we thus obtain Lie algebras to which the matrix Lie affgebras are retracted.

We use the following notation for classical (matrix) Lie algebras. The Lie algebra of $n\times n$-matrices over a field $\F$ is denoted by $\mathrm{gl}(n,\F)$ with $\mathrm{sl}(n,\F)$ standing for traceless matrices with entries from $\F$; $\mathrm{o}(n)$ denotes real anti-symmetric matrices, while $\mathrm{u}(n)$ stands for complex anti-hermitian matrices. Finally,  $\mathrm{su}(n)$ is the Lie algebra of all traceless anti-hermitian matrices.

\section{Lie brackets on affine spaces}
In a vector-free formulation (see e.g.\ \cite{BreBrz:hea}) an affine space over a field $\F$ can be realised as an algebraic system consisting of a set $A$ together with ternary operations
\begin{subequations}\label{heap.action}
\begin{equation}\label{heap}
\langle -,-,-\rangle : A\times A\times A \to A, \qquad (a,b,c)\mapsto \langle a,b,c\rangle,
\end{equation}
\begin{equation}\label{action}
  - \act_- - :    \F\times A\times A \to A, \qquad (\alpha,a,b)\mapsto \alpha \act _ab,
\end{equation}
\end{subequations}
satisfying the following conditions, for all $a,b,c,d,e\in A$ and $\alpha,\beta,\gamma \in \F$,
\begin{subequations}\label{aff.axioms}
    \begin{equation}\label{heap.axioms}
    \langle\langle a,b,c\rangle,d,e\rangle = \langle a,b,\langle c,d,e\rangle\rangle, \quad \langle a,b,b\rangle = a = \langle b,b,a\rangle, \quad \langle a,b,c\rangle = \langle c,b,a\rangle ,
        \end{equation}
        \begin{equation}\label{act.add}
       (\alpha -\beta+\gamma) \act_a b =  \langle \alpha\act_a b,   \beta\act_a b,  \gamma\act_a b \rangle  ,
        \end{equation}
        \begin{equation}\label{act.heap}
       \alpha \act_a \langle b,c,d\rangle =  \langle \alpha\act_a b,   \alpha\act_a c,  \alpha\act_a d \rangle  ,
        \end{equation}
        \begin{equation}\label{act.ass.uni.zero}
           (\alpha\beta)\act_a b = \alpha\act_a(\beta\act_a b), \quad 1\act_a b =b, \quad 0 \act_a b =a,
        \end{equation}
        \begin{equation}\label{act.base.change}
           \alpha\act_a b = \langle \alpha\act_cb, \alpha\act_c a, a\rangle.
        \end{equation}
\end{subequations}
A homomorphism of affine spaces or {\em affine transformation} is a function $f:A\to B$ that preserves both operations, that is, for all $a,b,c\in A$ and $\alpha \in \F$,
\begin{equation}\label{aff.tran}
    f \left( \langle a,b,c\rangle \right) =  \langle f(a),f(b),f(c)\rangle, \quad \mbox{and} \quad f(\alpha\act_a b) = \alpha\act_{f(a)}f( b).
\end{equation}

A few explanations and forays into terminology are in order now. The conditions \eqref{heap} mean that $A$ together with $\langle -,-,-\rangle$ is an {\em abelian heap}, the notion introduced by Pr\"ufer in \cite{Pru:the}. We will refer to $\langle -,-,-\rangle$ as the {\em heap operation}. The first condition in \eqref{heap.axioms} might be referred to as the {\em quasi-} or {\em heap associativity}, the second is known as the {\em Mal'cev identities}, and the last one is the commutativity of the heap operation (required originally in \cite{Pru:the}, but not needed for general heaps \cite{Bae:ein}). Altogether, the laws that heap operation of an abelian heap satisfies imply that the placement of brackets $\langle,\rangle$ does not play a role, hence we can write succinctly
$\langle a,b,c,d,e\rangle$ for $\langle \langle a,b,c\rangle,d,e\rangle$ or any other  distribution of heap operations. 

A homomorphism of heaps is a function that preserves the ternary operations, i.e.\ satisfies the first of condition \eqref{aff.tran}. Every abelian group $A$ is a heap in the canonical way with the heap operation, for all $a,b,c\in A$,
\begin{equation}\label{gr.he}
    \langle a, b ,c \rangle = a-b+c.
\end{equation}
Any additive map is automatically a heap homomorphism, but there are more heap homomorphisms than group homomorphisms, as any additive map composed with the translation by an element is also a heap homomorphism. 

Conversely, every (non-empty) abelian heap can be made into an abelian group or, to be precise, the family of isomorphic groups, by retracting the heap operation at any element $o\in A$. As a set, the group $G(A;o)$ is the same as $A$, with $o$ as the neutral element and the addition and inverse
\begin{equation}\label{retract}
    a+b = \langle a,o,b\rangle, \quad -a = \langle o,a,o\rangle.
\end{equation}
We refer to $G(A;o)$ as the {\em retract} of $A$ at $o$. The isomorphism between retracts of the same heap at different elements is given by the {\em translation automorphism},
\begin{equation}\label{trans}
    \tau_o^{\bar o}: G(A;o)\to G(A;\bar{o}), \qquad a\mapsto \langle a,o,\bar{o}\rangle .
\end{equation}
In view of all this a commutative heap operation can always be understood as a combination of binary operations as in \eqref{gr.he}.  Although the binary operations are defined relative to a chosen element, their combination is independent of the choice. This observation immediately justifies various equivalences such as
$$
\langle a,b,c\rangle =d \iff \langle a,d,c\rangle =b \iff \langle b,c,d\rangle =a
$$
and
$$
a=b \iff \langle a ,b, c\rangle = c, \qquad \mbox{for all $c\in A$}.
$$

The map $-\act_--$ is called the {\em action} of $\F$ on $A$ and  $a$ in $-\act_a-$ is known as the {\em base element}. Equation \eqref{act.add} means that the $-\act_ab: \F \to A$ is a heap homomorphism, and similarly, \eqref{act.heap} states the heap homomorpism property of $\alpha\act_a - : A\to A$. The first two conditions in \eqref{act.ass.uni.zero} state the associativity and unitality of the actions at a fixed base; the last normalisation property when combined with \eqref{act.add} means that $-\act_a b : \F \to G(A;a)$ is a group homomorphism. Finally, the {\em base change property} \eqref{act.base.change} ensures that all elements of $A$ play exactly the same role; every element can be a base for the action and the definition as a whole is independent of the choices made. This fits perfectly into the observer-independent framework alluded to in Introduction.

In the standard definition of an affine space as a set $A$ with a free and transitive action of a vector space $\overset{\to}{A}$, the heap operation and action are defined as, for all $a,b,c\in A$ and $\alpha \in \F$,
$$
\langle a,b,c\rangle = a+ \overset{-\!\to}{bc}, \qquad  
\alpha\act _ab = a+ \alpha\overset{-\!\to}{ab},
$$
where $\overset{-\!\to}{ab}$ denotes the vector from $a$ to $b$ etc.

Conversely, fixing any $o\in A$, the retract $G(A;o)$ becomes a vector space with multiplication by scalars
$$
\alpha a = \alpha\act_o b.
$$
We denote this vector space by $V(A;o)$. The vector from $a$ to $b$ is then given by the formula $\langle o,a,b\rangle = b-a$, where the difference is written in $G(A;o)$.

The major difficulty in defining Lie brackets on an affine space is to make all conditions independent from a specific element. A way to overcome this obstacle is presented in the following
\begin{definition}(\cite[Definition~3.1]{BrzPap:Lie})\label{def.Lie}
    Let $A$ be an affine space over $\F$. A {\em Lie bracket} on $A$ is a binary operation $[-,-]:A\times A \to A$ that satisfies the following conditions: 
    \begin{blist}
    \item for all $a\in A$, both $[a,-]$ and $[-,a]$ are affine transformations $A\to A$;
    \item affine antisymmetry, that is, for all $a,b\in A$,
    \begin{equation}\label{antisym}
        \langle [a,b],[a,a], [b,a]\rangle = [b,b];
    \end{equation}
    \item the affine Jacobi identity, that is, for all $a,b,c\in A$,
    \begin{equation}
        \langle [a,[b,c]],[a,a], [b,[c,a]], [b,b], [c,[a,b]] \rangle = [c,c].
    \end{equation}
    \end{blist}
    An affine space together with a Lie bracket is called a {\em Lie affgebra}.

    A homomorphism of Lie affgebras is an affine map that preserves the Lie bracket.
\end{definition}

Condition (a) means that the Lie bracket is {\em bi-affine transformation}. 

Basic examples of Lie affgebras include an affine space $A$ itself with the bracket 
\begin{equation}\label{zeta}
    [a,b]=\zeta\act_ab,
\end{equation} 
where $\zeta$ is any fixed scalar. As shown in \cite[Proposition~3.4]{BrzPap:Lie}, if dimension of $A$ is at least one, brackets corresponding to different scalars equip $A$ with non-isomorphic Lie affgebra structures. Furthermore, for any associative bi-affine multiplication on $A$, or more generally, for a bi-affine multiplication that satisfies conditions similar to that of a pre-Lie algebra one can define the {\em affine commutator} Lie bracket
\begin{equation}\label{com.Lie}
    [a,b] = \langle ab, ba,b\rangle;
\end{equation}
see \cite[Definition~3.8 \& Proposition~3.9]{BrzPap:Lie}.

It has been observed in \cite[Theorem~3.15]{BrzPap:Lie} that with every Lie affgebra $(A, [-,-])$ one can associate a family of isomorphic Lie algebras. Each member of this family has $V(A;o)$ as the underlying vector space and the Lie bracket
\begin{equation}\label{Lie.retract}
   [a,b]_o := \langle [a,b],[a,o],[o,o],[o,b],o\rangle =  [a,b]-[a,o]+[o,o]-[o,b].
\end{equation}
We denote this Lie algebra by $L(A;o)$ and call it the {\em Lie algebra retract} of $A$ at $o$. Different Lie affgebra structures might and often do collapse to the same Lie algebra. For example, for all $\zeta\in \F$ and $o\in A$, the Lie algebra associated to the bracket \eqref{zeta} is trivial, that is $[a,b]_o = o$. On the other hand, as explained in \cite[Theorem~4.3]{AndBrzRyb:ext}, any associative bi-affine multiplication  on $A$ makes $V(A;o)$ into an associative algebra with the product
$$
a\bullet b = ab - ao+o^2-ob.
$$
The affine commutator bracket \eqref{com.Lie}  on $A$ retracts at $o$ to  the bilinear commutator bracket $[a,b]_o = a\bullet b - b\bullet a$.

Finally, let us point out that if a Lie bracket on an affine space $A$ is an idempotent operation, that is $[a,a]=a$, for all $a\in A$, then, for all $o\in A$, $[a,b]_v := [a,b]-b$ is a vector-valued Lie bracket on the affine space viewed traditionally as a pair $(A, V(A;o))$, as introduced in \cite{GraGra:Lie}. This is the case for both \eqref{zeta} and \eqref{com.Lie} structures.

\section{Matrix Lie affgebras and corresponding Lie algebras}
In this paper we consider the following sets of matrices, which are defined for all positive integers $n$.

{\bf General normalised affine matrices:} 
$$
\gna(n,\F) = \{\mathbf{a}=(a_{kl})_{k,l=1}^{n+1} \in \mathrm{gl}(n+1, \F) \; |\; \sum_{k}a_{kl} = \sum_{k}a_{lk} =1,\; \forall l\}.
$$ 

{\bf Special normalised affine matrices:} 
$$
\sna(n,\F) = \{\mathbf{a} \in \gna(n,\F)  \; |\; \tr\mathbf{a} = 0\}.
$$ 

{\bf Anti-symmetric normalised affine matrices:} 
$$
\ona(n) = \{\mathbf{a} \in \gna(n,\R)  \; |\; a_{kk} = 1,\; a_{kl}= - a_{lk},\; \forall k\neq l\}.
$$ 

{\bf Anti-hermitian normalised affine matrices:} 
$$
\una(n) = \{\mathbf{a}=(a_{kl})_{k,l=1}^{n+1} \in \mathrm{gl}(n+1, \C) \; |\; \sum_{k}a_{kl} = \sum_{k}a_{lk} =i,\; a_{lm} = - \bar a_{ml}, \forall l,m\}.
$$ 

{\bf Special anti-hermitian normalised affine matrices:} 
$$
\suna(n) = \{\mathbf{a} \in \una(n)  \; |\;  \tr\mathbf{a} = 0\}.
$$ 

\begin{lemma}\label{lem.examples}
 Let $\mathfrak{a}$ denote any of the sets $ \gna(n,\F)$, $\sna(n,\F)$, $\ona(n)$, $\una(n)$ or $\suna(n)$. Then  $\mathfrak{a}$ is a Lie affgebra with the following operations:
 \begin{blist}
     \item the heap operation:
     $$\langle \mathbf{a}, \mathbf{b}, \mathbf{c}\rangle  = \mathbf{a} - \mathbf{b} + \mathbf{c},$$
     \item the affine $\F$-action
     $$
     \alpha\act_\mathbf{a}\mathbf{b}= \alpha \mathbf{b}- \alpha\mathbf{a}+ \mathbf{a},
     $$
     \item the affine commutator Lie bracket
     $$
     [\mathbf{a}, \mathbf{b}]= \mathbf{a}\mathbf{b} - \mathbf{b} \mathbf{a} + \mathbf{b},$$
 \end{blist}
 for all $\mathbf{a}, \mathbf{b}, \mathbf{c}\in \mathfrak{a}$ and $\alpha \in \F$.
 \end{lemma}
 \begin{proof}
     This is a straightforward calculation no different from the method of the proof that $\sna(n,\C)$ is a Lie affgebra used in  \cite[Proposition~4]{Brz:spe}.
 \end{proof}

 The main result of the note provides one with alternative descriptions of matrix Lie algebras in terms of classical (matrix) Lie algebras. 
 \begin{theorem}\label{thm.iso}
 For all $n\neq \mathrm{char}(\F),\mathrm{char}(\F)-1$, 
 there are Lie affgebra isomorphisms with the following Lie sub-affgebras of  $\mathrm{gl}(n+1,\F)$:  
$$ \gna(n,\F)\cong 
     \begin{pmatrix}
         0 & 0\cr 0 & 1 
     \end{pmatrix} + 
     \begin{pmatrix}
         \mathrm{gl}(n,\F) & 0\cr 0 & 0 
     \end{pmatrix}; $$
$$\sna(n,\F)\cong 
     \begin{pmatrix}
         -\frac 1n \mathrm{I}_n & 0\cr 0 & 1 
     \end{pmatrix} + 
     \begin{pmatrix}
         \mathrm{sl}(n,\F) & 0\cr 0 & 0 
     \end{pmatrix};$$
     $$\ona(n)\cong 
     \mathrm{I}_{n+1} + 
     \begin{pmatrix}
         \mathrm{o}(n) & 0\cr 0 & 0 
     \end{pmatrix};$$
     $$ \una(n)\cong 
     \begin{pmatrix}
         0 & 0\cr 0 & i 
     \end{pmatrix} + 
     \begin{pmatrix}
         \mathrm{u}(n) & 0\cr 0 & 0 
     \end{pmatrix};$$
     $$ \suna(n)\cong 
     \begin{pmatrix}
         -\frac in \mathrm{I}_n & 0\cr 0 & i 
     \end{pmatrix} + 
     \begin{pmatrix}
         \mathrm{su}(n) & 0\cr 0 & 0 
     \end{pmatrix},$$
     where $\mathrm{I}_m$ denotes the $m\times m$ identity matrix.
 \end{theorem}
 \begin{proof}
     In order to develop a universal argument that covers all cases listed above, let us fix an element $c\in \F$ and define 
     $$
      \ga_c(n,\F) = \{\mathbf{a}=(a_{kl})_{k,l=1}^{n+1} \in \mathrm{gl}(n+1, \F) \; |\; \sum_{k}a_{kl} = \sum_{k}a_{lk} =c,\; \forall l\}.
     $$
     For all $c\in \F$, $\ga_c(n,\F)$ is a Lie affgebra with the same affine and Lie algebra structures as in Lemma~\ref{lem.examples}. Furthermore, the first three Lie affgebras in the theorem are Lie sub-affgebras of $\gna(n,\F) = \ga_1(n,\F)$, while the last two are Lie sub-affgebras of $\ga_i(n,\C)$. In fact, for all $c,c'\neq 0$, $\ga_c(n,\F) \cong \ga_{c'}(n,\F)$ as Lie affgebras. Indeed, consider the following function 
     \begin{equation}\label{c-c'}
     f: \ga_c(n,\F) \to \ga_{c'}(n,\F), \qquad \mathbf{a}\mapsto \mathbf{a} + (c'-c)\mathrm{I}_{n+1}.   
     \end{equation}
     Being a combination of a linear and constant terms this is an affine isomorphism, which changes the normalisation of rows and columns in the desired way. Since the matrix that is a multiple of identity is central, one easily checks that, for all $\mathbf{a},\mathbf{b}\in \ga_c(n,\F)$
     $$
     f(\mathbf{a}\mathbf{b} - \mathbf{b}\mathbf{a} + \mathbf{b}) =  f(\mathbf{a})f(\mathbf{b}) - f(\mathbf{b})f(\mathbf{a})+ f( \mathbf{b}),
     $$
     i.e.\ that $f$ preserves the  affine commutator Lie brackets.

      Clearly, as affine spaces
     \begin{equation}\label{gac-gao}
         \ga_c(n,\F) = \mathbf{o} + \ga_0(n,\F),
     \end{equation}
     where $\mathbf{o} \in \ga_c(n,\F)$ is freely fixed. One easily checks that this identification is an isomorphism of Lie affgebras, provided that $\mathbf{o}$ commutes with all elements of $\ga_0(n,\F)$.
     
     First we show the existence of a similarity transformation which implements an isomorphism of affine spaces 
     \begin{equation}\label{ga}
     \ga_c(n,\F)\cong 
         \begin{pmatrix}
         \mathrm{gl}(n,\F) & 0\cr 0 & c 
     \end{pmatrix}.
     \end{equation}
     To this end, consider the $(n+1)\times (n+1)$-matrix
     \begin{equation}\label{mat.P}
         \mathbf{P} = 
         \begin{pmatrix}
             1 & 1 & 1 &\cdots & 1 & 1 \cr
             0 & 0 & 0 &\cdots & -1 & 1 \cr
             \cdots  & \cdots  & \cdots &\cdots & \cdots & \cdots \cr
             0 & 0 & -1 &\cdots & 0 & 1 \cr
             0 & -1 & 0 &\cdots & 0 & 1 \cr
             -1 & 0 & 0 &\cdots & 0 & 1 \cr
         \end{pmatrix}.
     \end{equation}
     If $n+1\neq \mathrm{char}(\F)$, $\mathbf{P}$ is invertible with the inverse
     $$
    \mathbf{P}^{-1} =  \frac 1{n+1} \begin{pmatrix}
             1 & 1 & 1 &\cdots & 1 & -n  \cr
             1 & 1 & 1 &\cdots & -n & 1  \cr
             \cdots  & \cdots  & \cdots &\cdots & \cdots & \cdots  \cr
             1 & 1 & -n &\cdots & 1 & 1  \cr
             1 & -n & 1 &\cdots & 1 & 1 \cr
             1 & 1 & 1& \cdots &1 & 1
         \end{pmatrix}.
        $$
    Since the last column of $\mathbf{P}$ and the last row of $\mathbf{P}^{-1}$ contain only identities (up to an overall normalisation factor in the latter case),
    \begin{equation}\label{c.conv}
        \mathbf{P} \begin{pmatrix} 0 & 0\cr 0 & c
        \end{pmatrix} \mathbf{P}^{-1} = \frac{1}{n+1} \begin{pmatrix}
            c & c & \cdots & c\cr
            c & c & \cdots & c\cr
            \cdots & \cdots & \cdots & \cdots\cr
            c & c & \cdots & c
        \end{pmatrix} \in \ga_c(n,\F).
    \end{equation}
        Second, note that the sums of the elements in each of the first $n$ columns of $\mathbf{P}$ and in each of the first $n$ rows of $\mathbf{P}^{-1}$ vanish. This, combined with the observation that in any matrix $\mathbf{a}\in \begin{pmatrix}
         \mathrm{gl}(n,\F) & 0\cr 0 & 0 
     \end{pmatrix}$  the last row and column contain zeros only, we immediately conclude that 
     \begin{equation}\label{ga.conv}
         \mathbf{P}\mathbf{a}\mathbf{P}^{-1} \in \ga_0(n,\F).
     \end{equation}
      Since both $\ga_0(n,\F)$ and $\begin{pmatrix}
         \mathrm{gl}(n,\F) & 0\cr 0 & 0 
         \end{pmatrix} \cong \mathrm{gl}(n,\F)$
         are $n\times n$ dimensional vector spaces, the similarity transformation establishes an isomorphism of vector spaces. In view of the identification \eqref{gac-gao} we obtain the required  isomorphism \eqref{ga} of affine spaces:
         \begin{equation}\label{iso}
             f: \ga_c(n,\F) \to \begin{pmatrix}
         \mathrm{gl}(n,\F) & 0\cr 0 & c 
     \end{pmatrix}, \qquad \mathbf{a}\mapsto \mathbf{P}^{-1}\mathbf{a}\mathbf{P}.
         \end{equation}
         The  map $f$ is given by conjugation by an invertible matrix, hence it preserves the Lie brackets in Lemma~\ref{lem.examples}. Finally, $\begin{pmatrix}
         0 & 0\cr 0 & c 
     \end{pmatrix}$ commutes with all elements of  $\begin{pmatrix}
         \mathrm{gl}(n,\F) & 0\cr 0 & 0 
     \end{pmatrix}$, its similarity transform commutes with all elements of $\mathrm{ga}_0(n,\F)$ and, consequently, the map $f$ is a Lie affgebra isomorphism. Therefore, the first assertion follows.
                 
         As the similarity transformation $f$ preserves traces, when restricted to $\sna (n,\F)$ the map $f$ gives the  isomorphism
         $$
         \begin{aligned}
             \begin{pmatrix}
         -\frac 1n \mathrm{I}_n & 0\cr 0 & 1 
     \end{pmatrix} + 
     \begin{pmatrix}
         \mathrm{sl}(n,\F) & 0\cr 0 & 0 
     \end{pmatrix} \longrightarrow & \mathbf{P}\begin{pmatrix}
         -\frac 1n \mathrm{I}_n & 0\cr 0 & 1 
     \end{pmatrix} \mathbf{P}^{-1} + \mathrm{sa}_0(n,\F)\\
     &= \frac 1n \begin{pmatrix}
         0 & 1 & 1 & \cdots & 1\cr
         1 & 0 & 1 & \cdots & 1\cr
         \cdots & \cdots & \cdots & \cdots & \cdots\cr
         1 &  1 & 1 & \cdots & 0 
     \end{pmatrix}  + \mathrm{sa}_0(n,\F)\\
     &= \sna(n,\F),
         \end{aligned}
         $$
         where $\mathrm{sa}_0(n,\F)$ denotes the Lie subalgebra of traceless matrices in  $\ga_0(n,\F)$.  Since 
         $\begin{pmatrix}
         -\frac 1n \mathrm{I}_n & 0\cr 0 & 1 
     \end{pmatrix}$ commutes with all elements of  $\begin{pmatrix}
         \mathrm{sl}(n,\F) & 0\cr 0 & 0 
     \end{pmatrix}$, its similarity transform commutes with all elements of $\mathrm{sa}_0(n,\F)$ and hence we obtain the required isomorphism of Lie affgebras. 
     
         For the remaining three assertions, $\F$ is either $\mathbb{R}$ or $\mathbb{C}$. Since $\mathbf{P}$ is a real matrix, the Gram-Schmidt orthonormalisation of its columns with respect to the Euclidean or dot product, starting with the last column $(1,1,\ldots, 1)^T$, produces the real unitary, hence orthogonal matrix $\mathbf{U}$ in which the sums of the elements in each column except the last one are zero. Explicitly,
         $$
         \mathbf{U} = \begin{pmatrix}
             \frac 1{\sqrt{(n+1)n}} &  \frac 1{\sqrt{n(n-1)}} & \frac 1{\sqrt{(n-1)(n-2)}} & \cdots & \frac 1{\sqrt{6}} & \frac 1{\sqrt{2}} & \frac 1{\sqrt{n+1}}\cr
             \frac 1{\sqrt{(n+1)n}} &  \frac 1{\sqrt{n(n-1)}} & \frac 1{\sqrt{(n-1)(n-2)}} & \cdots & \frac 1{\sqrt{6}} & -\frac 1{\sqrt{2}} & \frac 1{\sqrt{n+1}}\cr
             \frac 1{\sqrt{(n+1)n}} &  \frac 1{\sqrt{n(n-1)}} & \frac 1{\sqrt{(n-1)(n-2)}} & \cdots & - \sqrt{\frac 23} & 0 & \frac 1{\sqrt{n+1}}\cr
             \frac 1{\sqrt{(n+1)n}} &  \frac 1{\sqrt{n(n-1)}} & \frac 1{\sqrt{(n-1)(n-2)}} & \cdots & 0 & 0 & \frac 1{\sqrt{n+1}}\cr
             \cdots &  \cdots & \cdots & \cdots & \cdots &\cdots & \cdots\cr
             \frac 1{\sqrt{(n+1)n}} &  \frac 1{\sqrt{n(n-1)}} & - \sqrt{\frac {n-2}{n-1}} & \cdots & 0 & 0 & \frac 1{\sqrt{n+1}}\cr
             \frac 1{\sqrt{(n+1)n}} &  -\sqrt{\frac {n-1}{n}} & 0 & \cdots & 0 & 0 & \frac 1{\sqrt{n+1}}\cr
             -\sqrt{\frac {n}{n+1}} &  0 & 0 & \cdots & 0 & 0 & \frac 1{\sqrt{n+1}}
         \end{pmatrix}.
         $$
        The last column in $\mathbf{U}$ and hence the last row of $\mathbf{U}^{-1}= \mathbf{U}^\dagger$  are multiples of $(1,1,\ldots, 1)^T$. Therefore,  
         $\mathbf{U}\begin{pmatrix} 0 & 0\cr 0 & c
        \end{pmatrix} \mathbf{U}^{\dagger}$ is the same matrix as in \eqref{c.conv}. Similarly, the vanishing of the sums of the first $n$ columns of $\mathbf{U}$ and correspondingly first $n$ rows of $\mathbf{U}^{\dagger}$ allows one to conclude as in \eqref{ga.conv}. In consequence one obtains the unitary transformation
         \begin{equation}\label{unitary}
             g: \ga_c(n,\F) \to \begin{pmatrix}
         \mathrm{gl}(n,\F) & 0\cr 0 & c 
     \end{pmatrix}, \qquad \mathbf{a}\mapsto \mathbf{U}^{\dagger}\mathbf{a}\mathbf{U},
         \end{equation}
which covers all the remaining assertions of the theorem. 
\end{proof}

Theorem~\ref{thm.iso} enables one easily to identify  (up to isomorphism) Lie algebra retracts of all the matrix Lie affgebras listed there. 
\begin{corollary}\label{cor.iso}
   The following table lists Lie algebra retracts of matrix Lie affgebras:
   \begin{center}
       \begin{tabular}{|c|c|}
       \hline
        matrix Lie affgebra    & Lie algebra retract \\
        \hline
          $\gna (n,\F)$  &  $\mathrm{gl}(n,\F)$ \\
          $\sna (n,\F)$ &  $\mathrm{sl}(n,\F)$\\
          $\ona (n)$ & $\mathrm{o}(n)$\\
        $\una (n)$ & $\mathrm{u}(n)$\\
        $\suna (n)$ & $\mathrm{su}(n)$\\
        \hline
       \end{tabular}
   \end{center}
\end{corollary}
\begin{proof}
    In view of Theorem~\ref{thm.iso}, Lie affgebras in the left column are isomorphic with Lie affgebras of the form
    $$
    \mathfrak{a} = \mathbf{o} + \begin{pmatrix}
        \mathfrak{g} & 0\cr 0 & 0
    \end{pmatrix},
    $$
    where $\mathfrak{g}$ is the corresponding Lie algebra from the right column of the table. When retracted at $\mathbf{o}$,
    $$
    L(\mathfrak{a};\mathbf{o}) = \begin{pmatrix}
        \mathfrak{g} & 0\cr 0 & 0\end{pmatrix} \cong \mathfrak{g}
    $$
    with the usual matrix vector space structure and with the commutator as the Lie bracket.
\end{proof}

\section*{Acknowledgements} 

The research of Tomasz Brzezi\'nski is partially supported by the National Science Centre, Poland, grant no. 2019/35/B/ST1/01115.

\bibliographystyle{amsplain}

\end{document}